\documentclass{compositio}
\usepackage[centertags]{amsmath}
\usepackage{amsfonts}
\usepackage{graphicx}
\usepackage{amssymb}
\usepackage{amsthm}
\usepackage{newlfont}
\usepackage[all]{xy}
 \newtheorem{thm}{Theorem}[section]
 \newtheorem{cor}[thm]{Corollary}
 \newtheorem{lem}[thm]{Lemma}
 
 \theoremstyle{definition}
 \newtheorem{defn}[thm]{Definition}
 \newtheorem{conj}[thm]{Conjecture}
 \theoremstyle{remark}
 \newtheorem{rem}[thm]{Remark}
 \newtheorem{ex}[thm]{Example}
 \numberwithin{equation}{section}
 \newcommand{\Spec}{\operatorname{Spec}}

 \newcommand{\Rat}{\operatorname{Rat}}

 \newcommand{\PGL}{\operatorname{PGL}}
 \newcommand{\Gal}{\operatorname{Gal}}

\newlength{\defbaselineskip}
\begin{document}

\title{An Algebraic Proof of Thurston's Rigidity for Maps With a Superattracting Cycle}

\author{Alon Levy}


\begin{abstract}We study rational self-maps of $\mathbb{P}^{1}$ whose critical points all have finite forward orbit. Thurston's rigidity theorem states that outside
a single well-understood family, there are finitely many such maps over $\mathbb{C}$ of fixed degree and critical orbit length. We provide an algebraic proof of
this fact for tamely ramified maps for which at least one of the critical points is periodic. We also produce wildly ramified counterexamples.\end{abstract}

\maketitle

\section{Introduction}

The behavior of a rational self-map of $\mathbb{P}^{1}$ of degree $d \geq 2$ at its critical points plays an important role in its global dynamics. Specifically, we
study a special set of maps:

\begin{defn}\label{PCF}A rational map $\varphi:\mathbb{P}^{1} \to \mathbb{P}^{1}$ is \textbf{postcritically finite} (PCF) if all of its critical points have finite
forward orbit.\end{defn}

PCF maps have attracted some attention from complex dynamists~\cite{Koc, BEKP}, who relate their combinatorial properties to their dynamical properties. At the same
time, arithmetic dynamists have studied their special Galois orbit properties. In brief, if we fix a point $z \in \mathbb{P}^{1}(K)$ for some arithmetically
interesting field $K$, the absolute Galois group $\Gal(\overline{K}/K)$ will act on the infinite tree of preimages of $z$. Conjecturally the Galois action has more
or less full image, but if $\varphi$ is PCF, then the image is much smaller~\cite{AHM, BJ1, BJ2}. If one views rational maps as analogous to elliptic curves, as is
the approach used in~\cite{ADS}, then PCF maps are thus somewhat analogous to elliptic curves with complex multiplication.

Our goal in this paper is to study the PCF maps from the point of view of the moduli space of rational maps on $\mathbb{P}^{1}$. We write rational maps by their
coordinates, $$\varphi(z) = \frac{f(z)}{g(z)} = \frac{a_{d}z^{d} + \ldots + a_{0}}{b_{d}z^{d} + \ldots + b_{0}}$$ The space of morphisms is a subset of the space
$$(a_{d}:\ldots:a_{0}:b_{d}:\ldots:b_{0}) \in \mathbb{P}^{2d+1}$$ defined by the open affine condition that $\gcd(f, g) = 1$; we call this open affine subspace
$\Rat_{d}$. The geometry of $\varphi$ is preserved under coordinate-change; in particular, if $A \in \PGL(2)$ and $\varphi$ is PCF, then so is its conjugation
$A\varphi A^{-1}$. Thus we need to quotient the space by $\PGL(2)$-conjugation. We set $\Rat_{d}//\PGL(2) = \mathrm{M}_{d}$; as established in~\cite{Sil96, PST,
Lev1}, the quotient $\mathrm{M}_{d}$ is geometric and the stabilizer groups in $\PGL(2)$ are finite.

In particular, $\dim\mathrm{M}_{d} = 2d-2$. Conversely, when $\varphi$ is tamely ramified it has $2d-2$ critical points, counted with multiplicity. Although the
condition that $\varphi$ is PCF cannot be expressed as a finite number of equations, if we specify the size of the forward orbit of each critical point, then we
obtain an algebraic equation. Hence we obtain $2d-2$ equations. It is therefore natural to ask if those equations always intersect in a finite number of points.

Unfortunately, in one special case, the $2d-2$ equations will intersect in a curve, rather than in finitely many points:

\begin{defn} The map $\varphi: \mathbb{P}^{1} \to \mathbb{P}^{1}$ is called a \textbf{Latt\`{e}s map} if there is an elliptic curve $E$, a morphism $\alpha:E\to E$,
and a finite separable map $\pi$ such that the following diagram commutes:

\begin{equation*}
  \xymatrix@R+2em@C+2em{
  E \ar[r]^{\alpha} \ar[d]_{\pi} & E \ar[d]^{\pi} \\
  \mathbb{P}^{1} \ar[r]_{\varphi} & \mathbb{P}^{1}
  }
\end{equation*}

If we choose $\pi$ to be the projection by $P\sim -P$, then $\alpha$ must be of the form $\alpha_{m, T}: P \mapsto mP + T$ where $T \in E[2]$.\end{defn}

The points of a Latt\`{e}s map $\varphi_{E, \alpha}$ with finite forward image are precisely the points that come from the torsion points on $E$. Moreover, because
$\alpha$ is unramified, the critical points come from the critical points of $\pi$, and those are necessarily torsion points; hence, all Latt\`{e}s maps are PCF.
Conversely, Latt\`{e}s maps $\varphi_{E, \alpha}$ with a fixed (more precisely, continuously varying) $\alpha$ form a curve in $\mathrm{M}_{m^{2}}$ according to the
$j$-invariants of $E$. For a more complete treatment, see~\cite{ADS}.

However, as shown by Thurston, the Latt\`{e}s maps are the only counterexample to the expectation that the equations defining PCF maps intersect in finitely many
points. Thurston studied PCF maps based on underlying combinatorial data:

\begin{defn}The \textbf{critical portrait} of a PCF map $\varphi$ is a directed graph on its critical points and their forward images in which the edge $x \mapsto
y$ occurs if and only if $\varphi(x) = y$. This includes the case in which $x = y$.\end{defn}

All Latt\`{e}s maps with the same $\alpha$ have the same critical portrait. We have,

\begin{thm}(Thurston's Rigidity~\cite{DH, BBLPP}) Over $\mathbb{C}$, each critical portrait, except the portraits defining the Latt\`{e}s curves, admits only
finitely many maps $\varphi \in \mathrm{M}_{d}$, which are defined over $\overline{\mathbb{Q}}$; moreover, the intersection of the equations defining the portrait
is a reduced scheme, i.e. the equations intersect transversely.\end{thm}

The importance of Thurston's rigidity goes beyond PCF maps. Indeed, it is useful when discussing the general behavior of periodic points of $\varphi$, that is
points $z$ for which $\varphi^{n}(z) = z$. First, we recall:

\begin{defn}The \textbf{multiplier} of a map $\varphi$ at a point $z$ of period $n$, that is a point for which $\varphi^{n}(z) = z$, is the eigenvalue of the
induced action on tangent spaces, $(\varphi^{n})'(z)$; it is conjugation-invariant, and invariant for all points in a single periodic cycle.\end{defn}

\begin{defn}In analogy with the terminology of attracting and repelling periodic points when $\varphi$ is defined over a valued field, we say a periodic cycle is
\textbf{superattracting} if its multiplier is $0$, an algebraic rather than an analytic condition. Observe that a cycle is superattracting if and only if it
contains a critical point.\end{defn}

Epstein~\cite{Eps2} uses rigidity to prove a bound on the number of periodic cycles whose multipliers are in the unit disk. In addition, McMullen~\cite{McM} uses
rigidity to prove a deep structure result on $\mathrm{M}_{d}(\mathbb{C})$:

\begin{thm}(McMullen) Let $\Lambda_{n}: \mathrm{M}_{d} \to \mathbb{A}^{k_{n}}$ be the map sending $\varphi$ to the symmetric functions in the period-$n$
multipliers; here $k_{n}$ is just the dimension of the target space. For sufficiently large $n$, $$\Lambda_{1} \times \ldots \times \Lambda_{n}:
\mathrm{M}_{d}(\mathbb{C}) \to \mathbb{A}^{k_{1} + \ldots + k_{n}}(\mathbb{C})$$ is finite-to-one away from the Latt\`{e}s curves.\end{thm}

Since the space $\mathrm{M}_{d}$ is defined over $\mathbb{Z}$, and the PCF locus is defined over $\mathbb{Z}$ as well, it is reasonable to ask if rigidity and
McMullen's result can be extended to arithmetically interesting fields. Although they extend to number fields by the Lefschetz principle, the proofs of both
rigidity and the derivation of McMullen's result from rigidity employ transcendental techniques, and therefore do not extend to characteristic $p$. Finding
arithmetic analogs of both results is an active research topic in arithmetic dynamics, with some partial results due to Epstein~\cite{Eps}, who uses heights to
prove rigidity for a special class of polynomial maps, and Hutz and Tepper~\cite{HT}, who prove that McMullen's result is true for a generic polynomial over
$\mathbb{Z}$ and compute the degree of the finite-to-one map.

Upon looking at various wildly ramified cases, the author's hopes that the two results would be true verbatim in characteristic $p$ were dashed. For example, if $k$
is a field of characteristic $p$, then the curve $$\varphi_{t}(z) = z^{mp} + tz, t \in k$$ is a non-isotrivial curve in $\mathrm{M}_{mp}$ since the multiplier at
$0$ varies, but is PCF since the only critical point is $\infty$ and is clearly fixed. In general, if $p < d$ and $p \nmid d$ then $z^{d} + tz^{p}$ only has $0$ and
$\infty$ for critical points and both are fixed, and can be seen to be a curve in $\mathrm{M}_{d}$ by direct examination of the action of each $A \in \PGL(2)$ on
the coefficients.

For a counterexample to McMullen, we can take any sufficiently high-dimension family with constant derivative, i.e. $\varphi(z) = cz + f(z^{p})/g(z^{p})$ where $c$
is a constant, which is again wildly ramified. For this family to not map to a single point in $\mathrm{M}_{d}$, we require $f(z)/g(z)$ to have degree $2$ or higher
and thus can be used when $p \mid d$ and $p < d$, but not when $p = d$, and indeed Silverman~\cite{Sil96} proved that $\Lambda_{1}$ is an isomorphism from
$\mathrm{M}_{2}$ to $\mathbb{A}^{2}$ over $\mathbb{Z}$.

If we think of rigidity as the statement that $2d-2$ equations over a space of dimension $2d-2$ should intersect in finitely many points, then its failure in the
wildly ramified case is not surprising, since there are fewer than $2d-2$ critical points even counted with multiplicity. It is therefore reasonable to propose,

\begin{conj}\label{arith}Except for the Latt\`{e}s curve, the finiteness portion of rigidity holds in all tamely ramified cases.\end{conj}

In positive characteristic, transversality fails, even in large characteristic. Indeed, for quadratic polynomials of the form $z^{2} + c$, the critical point is
$0$, and the condition that the critical point has exact period $3$ is $c^{3} + 2c^{2} + c + 1$, a polynomial of discriminant $-23$. It is an open question whether,
just for quadratic polynomials, the set of primes for which transversality fails has zero density. For our purposes the question then is whether we can prove
finiteness, and in the sequel, we will take rigidity to mean the finiteness portion of Thurston's rigidity.

It is currently beyond the author's means to prove Conjecture~\ref{arith}. The proof of rigidity in~\cite{DH} is topological, and has portions that cannot yet be
generalized beyond $\mathbb{C}$. However, in some cases, a direct algebraic proof is feasible. The case of polynomials is already known:

\begin{thm}\label{poly}\cite{Lev3} Rigidity holds for polynomials whenever their reduction after any coordinate change is still tamely ramified, which is the case
if they are of degree $d < p$ or if they are compositions or iterates of polynomials of degree less than $p$.\end{thm}

Ingram argues in private communication that a modification of the argument in~\cite{Ing} could produce a second proof of Theorem~\ref{poly} over a finite field of
characteristic $p > d$ using his theory of critical heights of polynomials.

We will prove a stronger result, using purely algebraic methods:

\begin{thm}\label{main}Rigidity holds for tamely ramified rational maps with a superattracting cycle.\end{thm}

In Section~\ref{finiteness}, we will prove Theorem~\ref{main} by looking at a grand coordinate ring that encodes not only the coefficients $a_{i}$ and $b_{i}$ of
each rational map $\varphi$ satisfying the hypotheses of the theorem, but also the critical points and the critical values. We turn this into a graded ring by
assigning each critical point weight $1$ and assigning other variables weights that make the defining equations for $\varphi$ homogeneous. We then show that instead
of looking at the system of equations $\varphi^{n}(\zeta_{i}) = \varphi^{m}(\zeta_{i})$ for all critical points $\zeta_{i}$ (here, $n$ and $m$ are fixed integers,
and for fixed $n$ and $m$ we have just finitely many critical portraits), we can look only at the top-degree portions of those equations, which are much simpler. A
series of reductions then converts rigidity into a problem of parametrizing rational maps by their critical values, a solved problem using the theory of branched
covers.

What we actually prove is not a statement about geometric loci of PCF maps, but about the corresponding rings. We will actually prove that if we take a carefully
chosen slice $X$ of $\Rat_{d}$ with $\infty$ a fixed critical point such that $X$ maps finite-to-one into $\mathrm{M}_{d}$, then a coordinate ring that is a finite
extension of $\mathcal{O}_{X}$ modulo the PCF ideal $\varphi^{n}(\zeta_{i}) = \varphi^{m}(\zeta_{i})$ yields a ring with only finitely many possibilities for the
critical value set. Thus there are only finitely many tamely ramified morphisms in the ring.

This method naturally fails in the two known exceptions for rigidity: it fails in the Latt\`{e}s case because the Latt\`{e}s maps' critical points are preperiodic
but never periodic, and it fails in wildly ramified cases because multiple parts of the proof require tame ramification for finiteness to work. A full
generalization to tamely ramified non-Latt\`{e}s maps is unlikely. Unless we can produce a critical point that maps to itself under some iterate of $\varphi$, the
equations $\varphi^{n}(\zeta_{i}) = \varphi^{m}(\zeta_{i})$ are already homogeneous with respect to the grading we use, and there is no hope of simplifying them.

We also investigate integrality, a related result, first investigated by Epstein in~\cite{Eps} as a way of proving rigidity for polynomials of prime-power degree.
Instead of proving the finiteness of a module over a field $k$, we can prove the finiteness of a module over a ring, say a local ring $\mathcal{O}_{K}$. This turns
out to be equivalent to the statement that the set of critical values of a PCF map is integral. In Section~\ref{integrality}, we prove,

\begin{cor}\label{integralmain}Let $\varphi(z)$ be a PCF map with a superattracting cycle defined over a local field, such that any integral model of $\varphi(z)$
has tamely ramified reduction. Then the configuration of the critical values of $\varphi(z)$ is integral.\end{cor}

Only this definition of integrality is true. A PCF map satisfying the hypotheses of Corollary~\ref{integralmain} does not have to be integral itself, i.e. have good
reduction; only its critical value set must be integral. In personal communication Rivera-Letelier provided a PCF map satisfying the hypothesis of
Corollary~\ref{integralmain} that nonetheless has bad reduction (Example~\ref{jrl}).

A separate approach to proving rigidity is to use McMullen's theorem. McMullen's proof is also transcendental, but it appears slightly more amenable to attacking in
positive characteristic using non-archimedean analysis. We also have,

\begin{thm}\label{auto}~\cite{Lev3} Let $K$ be a function field of characteristic $p$. If $\varphi \in \mathrm{M}_{d}(K)$ is PCF and $p > d$ or $\varphi$ is the
composition of maps of degree less than $p$ then its multipliers lie in the field of constants; in other words, McMullen implies rigidity.\end{thm}

Since in~\cite{Sil96} it is shown that $\Lambda_{1}$ is an isomorphism over $\mathbb{Z}$ between $\mathrm{M}_{2}$ and a plane in $\mathbb{A}^{3}$, we immediately
obtain,

\begin{cor}\label{sil}Thurston's rigidity is true in $\mathrm{M}_{2}$ in any characteristic except $2$.\end{cor}

Theorem~\ref{auto} is true slightly more generally: if $p \leq d$, it holds as long as $\varphi$ satisfies the same condition on tamely ramified reductions as in
Corollary~\ref{integralmain}. Unfortunately, even this approach is problematic as a means of proving rigidity, since it seems easier to prove rigidity directly than
to prove it via McMullen. The results of Theorem~\ref{main} and Corollary~\ref{sil} remain the strongest positive-characteristic rigidity results known to the
author.

\section{Finiteness}\label{finiteness}

We write $\varphi$ for a rational function of degree $d$ with a multiplicity-$e$ pole at $\infty$, where $e > 1$. We write $\varphi$ as $f/g$, with $f = \sum
a_{i}z^{i}$ and $g = \sum b_{i}z^{i}$. We assume $0$ is a fixed point, i.e. $a_{0} = 0$; we lose nothing by making this assumption since $\varphi$ has more than one
fixed point (the assumption that $e > 1$ means the multiplier at $\infty$ is $0$, whereas a unique fixed point has multiplier $1$). We may make one final
conjugation, which we use to assume $a_{d} = b_{d-e}$; we normalize to assume $a_{d} = b_{d-e} = 1$. We also assume that the characteristic does not divide any of
the ramification degrees. With the above assumptions that $a_{0} = 0, a_{d} = b_{d-e} = 1$, we call the resulting space of maps with a degree-$e$ pole at infinity
$X \subseteq \Rat_{d}$. Note that $X$ maps finite-to-one into $\mathrm{M}_{d}$.

\begin{lem}\label{crit}In $X$, the configuration of critical values of $\varphi$ is enough to specify each tamely ramified $\varphi$ up to finitely many
choices.\end{lem}

\begin{proof}This follows from standard results on monodromy theory. See~\cite{Sza} for more details. In brief, the tame absolute Galois group of $\mathbb{P}^{1}$
minus $n$ points, $\pi_{1}^{p'}$, is the profinite prime-to-$p$ completion of the free group on $n-1$ generators (Theorem 4.9.1), and finite unramified covers
correspond to finite sets with a continuous left-action of the absolute Galois group (Theorem 4.6.4); for tame ramification, we also require the inertia group at
each generator to be prime to $p$. Now specifying the list of critical values with multiplicities bounds the degree of the cover, bounding the size of those finite
sets, and there are finitely many continuous actions of a topologically finitely generated group on a finite set.

Covers of $\mathbb{P}^{1}$ are only unique up to the right-action of $\PGL(2)$. While this is not the same as the conjugation action we use to obtain
$\mathrm{M}_{d}$, $X$ maps finite-to-one into the quotient by the right-action as well. This is because we may pick $\infty$ to be an order-$e$ pole and $0$ to be a
zero, involving just finitely many choices, and then up to a finite quotient the only $\PGL(2)$ right-action that preserves this choice is the diagonal action. But
now the diagonal action of multiplication by $t$ acts on $a_{d}/b_{d-e}$ as multiplication by $t^{-e}$, and so we again get just finitely many choices that preserve
$X$.\end{proof}

On the affine variety $X$ we would like to show that the equations defining a PCF locus intersect at finitely many points. This is equivalent to showing that the
associated coordinate ring is finite as a $k$-module. More precisely, let $\zeta_{i}$ for $i = 1, \ldots, 2d-1-e$ be the critical points and $\xi_{i}$ be the
associated critical values.

Now let $$R = k[\zeta_{1}, \ldots, \zeta_{2d-1-e}, \xi_{1}, \ldots, \xi_{2d-1-e}, a_{d-1}, \ldots, a_{1}, b_{d-e-1}, \ldots, b_{0}]$$ We regard it as a graded ring,
with $$\deg \zeta_{i} = 1, \deg \xi_{i} = e, \deg a_{i} = d-i, \deg b_{i} = d-e-i$$ This requires us to mark the critical points, but we lose nothing from doing so,
since there are only finitely many of them, which does not change any finiteness question. We impose the usual relations of coordinates, critical points, and
critical values; we call the ideal they generate $I_{Y}$ and the quotient ring $\mathcal{O}_{Y}$. The variety $Y = \Spec\mathcal{O}_{Y}$ is an $S_{2d-1-e}$-cover of
a variety that contains $X$ as an open dense subvariety but also degenerate maps for which $f$ and $g$ have a common root.

\begin{lem}\label{grade}The ring $\mathcal{O}_{Y}$ is also graded with the same grading as $R$.\end{lem}

\begin{proof}It suffices to show that $I_{Y}$ is generated by homogeneous polynomials. Now the coordinate-critical point relations are encoded by the critical point
polynomial $$\sum_{(i, j) = (1, 0)}^{(d, d-e)}(i-j)a_{i}b_{j}z^{i+j-1} = e\prod_{i = 1}^{2d-1-e}(z - \zeta_{i})$$ which imposes $2d-1-e$ equations, one for each
coefficient except the leading term. Now on the left, the $z^{k}$-coefficient is the sum of $a_{i}b_{j}$'s for which $i+j = k+1$ and thus its degree is $2d-e-k-1$,
while on the right it is $$\sigma_{2d-e-k-1} = \sum_{l_{1} < \ldots < l_{2d-e-k-1}}\zeta_{l_{1}}\ldots\zeta_{l_{2d-e-k-1}}$$ which again has degree $2d-e-k-1$.
Since we assume tame ramification throughout, $e \neq 0$ and the right-hand side is not zero. We're now equating homogeneous algebraic expressions of the same
degree, so we obtain homogeneous equations.

The critical point-critical value relations are of the form $\varphi(\zeta_{i}) = \xi_{i}$. By symmetry, it suffices to show that the relation is homogeneous for
$\zeta_{1}$ and $\xi_{1}$. Clearing denominators, we obtain $f(\zeta_{1}) = \xi_{1}g(\zeta_{1})$. Since $\deg\xi_{1} = e$, it suffices to show $f(\zeta_{1})$ is
homogeneous of degree $d$ and $g(\zeta_{1})$ is homogeneous of degree $d-e$. But now each $a_{i}\zeta_{1}^{i}$ term has degree $d$ and each $b_{j}\zeta_{1}^{j}$
term has degree $d-e$, and we are done.\end{proof}

Lemma~\ref{crit} simply says that if $S$ is the $k$-algebra generated by the $\xi_{i}$s alone, with a natural inclusion $S \to R$, then there is a finite module
basis for $\mathcal{O}_{Y}$ that generates all well-defined tamely ramified morphisms in $\mathcal{O}_{Y}$ over $S/I_{Y}\cap S$. This is false if we drop either of
the conditions: wildly ramified maps have fewer critical points so there are positive-dimension families with constant critical values, for example $z^{p} + tz$ is
a non-isotrivial family whose only critical point and value is $\infty$; in addition, if we allow the maps to degenerate, then each common root of $f$ and $g$ is a
critical point and any point in $\mathbb{P}^{1}$ can be set as its corresponding critical value without violating the equations defining $Y$.

We will also use,

\begin{defn}Let $h \in R$ or $h \in \mathcal{O}_{Y}$. If $\deg h = D$, we say the \textbf{top-degree term} of $h$ is its degree-$D$ part. In particular, it is not
necessarily a monomial.\end{defn}

For a fixed pair of integers $n > m$, we need to show that there are finitely many tamely ramified maps in $X$ for which $\varphi^{m}(\zeta_{i}) =
\varphi^{n}(\zeta_{i})$ for all $i$. Thus Theorem~\ref{main} is equivalent to showing finiteness on the level of rings and modules:

\begin{thm}\label{isovalue}Let $M = \mathcal{O}_{Y}/I_{m, n}$, where the ideal $I_{m, n}$ is generated by $\varphi^{m}(\zeta_{i}) = \varphi^{n}(\zeta_{i})$ for all
$i$. Then the image ring $M'$ of the map from $S$ to $M$ via $R$ and $\mathcal{O}_{Y}$ is a finite $k$-module.\end{thm}

\begin{proof}It suffices to show that there exists a finite $k$-submodule of $R$, which by abuse of notation we also call $M$, such that each $h \in R$ can be
written as $s + r$ where $s \in I_{m, n}$ and $r \in M$. Since the set $$\{x \in R: \deg x \leq D\}$$ is finite for each integer $D$, this is equivalent to proving
that for some $D$, each $h \in R$ can be written as $s + r$ where $s \in I_{m, n}$ and $\deg r \leq D$.

Moreover, showing that $h$ can be written as $s + r$ with $\deg r \leq D$ is equivalent to showing that whenever $\deg h > D$, we can write $h$ as $s + h_{1}$ with
$\deg h_{1} < \deg h$. In one direction the equivalence is obvious; in the other direction, if $h = s + h_{1}$ and $\deg h_{1} > D$ then we can write $h_{1}$ as
$s_{1}+ h_{2}$ where $s \in I_{m, n}$ and $\deg h_{2} < \deg h_{1}$, and after finitely many steps we will obtain $h_{i}$ such that $\deg h_{i} \leq D$.

The equation $h = s + h_{1}$ can also be viewed in terms of top-degree terms. By definition, $\deg h_{1} < \deg h$, and so $h_{1}$ has no contribution to the
top-degree term of the equation. Any lower-degree term in $h$ or $s$ can also be added back to $h_{1}$, and therefore if the equation is satisfied for top-degree
terms only, it can be satisfied in general. Put another way, it suffices to show that given a homogeneous $h$ of degree more than $D$, we can find $s \in I_{m, n}$
such that the top-degree term of $s$ is exactly $h$.

We replace $I_{m, n}$ with $I'_{m, n}$, whose generators are the top-degree terms of the generators of $I_{m, n}$. It suffices to show that $h$ lies in this ideal:
to see why, write the generators of $I_{m, n}$ as $s_{1}, \ldots, s_{2d-1-e}$, and write their top-degree terms as $s'_{1}, \ldots, s'_{2d-1-e}$, so that $I'_{m, n}
= (s'_{1}, \ldots, s'_{2d-1-e})$. If $h \in I'_{m, n}$, then we can find homogeneous elements in $\mathcal{O}_{Y}$, $t_{1}, \ldots, t_{2d-1-e}$, such that $h = \sum
s'_{i}t_{i}$, and then $\sum s_{i}t_{i}$ has top-degree term equal to $h$.

Observe now that we can work backward: showing that $h \in I'_{m, n}$ whenever $h$ is homogeneous of degree more than $D$ is equivalent to showing that
$\mathcal{O}_{Y}/I'_{m, n}$ is a finite module. This is because, since $I'_{m, n}$ and $h$ are both homogeneous, the $h = s + r$ condition reduces to $h = s$.
Likewise, we can work backward and use the fact that the map $Y \to X$ is finite to show only that there are finitely many tamely ramified morphisms $\varphi \in Y$
satisfying the equations of $I'_{m, n}$.

Let us now compute the top-degree term of the equation $\varphi^{m}(\zeta_{i}) = \varphi^{n}(\zeta_{i})$. Write $\varphi^{n}$ as $f_{n}/g_{n}$, such that $$f_{n}(z)
= \sum a_{i}(f_{n-1}(z))^{i}(g_{n-1}(z))^{d-i} \quad \mbox{and} \quad g_{n}(z) = \sum b_{i}(f_{n-1}(z))^{i}(g_{n-1}(z))^{d-i}$$ By convention, $f_{1} = f$ and
$g_{1} = g$, or alternatively $f_{0}(z) = z$ and $g_{0}(z) = 1$. Then we have:

\begin{lem}\label{topdegree}For each expression $z \in R$ such that $\deg z > 1$, we have $\deg f_{n}(z) = d^{n}\deg z$ and $\deg g_{n}(z) = (d^{n} - e^{n})\deg z$.
If the top-degree term of $z$ is equal to $z'$, then the top-degree term of $f_{n}(z)$ is $z'^{d^{n}}$ and the top-degree term of $g_{n}(z)$ is
$z'^{d^{n}-e^{n}}$.\end{lem}

\begin{proof}We use induction on $n$. If $n = 1$, then $f_{n} = f$ and $g_{n} = g$. In the sum $$\sum_{i = 1}^{d} a_{i}z^{i}$$ the degree of the $i$th term is $d-i
+ i\deg z$ and since $\deg z > 1$, this degree is strictly maximized when $i = d$ and the degree is equal to $d\deg z$. Likewise, in the sum $$\sum_{i = 0}^{d-e}
b_{i}z^{i}$$ the degree of the $i$th term is $d-e-i + i\deg z$ and this is maximized when $i = d-e$ and the degree is equal to $(d-e)\deg z$. In both cases, letting
$z'$ be the top-degree term of $z$, the top-degree term of $f(z)$ is the top-degree term of $z^{d}$, which is just $z'^{d}$, and likewise the top-degree term of
$g(z)$ is $z'^{d-e}$.

Now, suppose the lemma is true up to $n-1$. Then in the sum $$\sum_{i = 1}^{d} a_{i}(f_{n-1}(z))^{i}(g_{n-1}(z))^{d-i}$$ the $i$th term has degree

\begin{align*}
d-i + i\deg f_{n-1}(z) + (d-i)\deg g_{n-1}(z) &= d-i + id^{n-1}\deg z + (d-i)(d^{n-1} - e^{n-1})\deg z\\
&= d + d(d^{n-1} - e^{n-1})\deg z + i(e^{n-1}\deg z - 1)
\end{align*}

This is (strictly) maximized when $i$ is maximized, i.e. when $i = d$ and the degree is $d^{n}\deg z$. Thus the top-degree term comes only from $(f_{n-1}(z))^{d}$,
and by the induction hypothesis it is equal to $z'^{(d^{n-1})d}$.

In the denominator, in the sum $$\sum_{i = 0}^{d-e} b_{i}(f_{n-1}(z))^{i}(g_{n-1}(z))^{d-i}$$ the degree of the $i$th term is $d-e-i + i\deg f_{n-1}(z) + (d-i)\deg
g_{n-1}(z)$ which is also maximized when $i$ is maximized, i.e. when $i = d-e$, and then the degree is $(d^{n} - e^{n})\deg z$. As the top-degree term comes only
from $(f_{n-1}(z))^{d-e}(g_{n-1}(z))^{e}$, by the induction hypothesis it is equal to $z'^{d^{n-1}(d-e) + (d^{n-1} - e^{n-1})e}$, which simplifies to $z'^{d^{n} -
e^{n}}$ as required.\end{proof}

We can apply Lemma~\ref{topdegree} directly to $\xi_{i}$ and rewrite $\varphi^{m}(\zeta_{i}) = \varphi^{n}(\zeta_{i})$ as $\varphi^{m-1}(\xi_{i}) =
\varphi^{n-1}(\xi_{i})$. Alternatively, when $z = \zeta_{i}$, $f(z)$ is homogeneous of degree $d$ and $g(z)$ is homogeneous of degree $d-e$ and neither identically
vanishes because $\xi_{i}$ is not identically $0$ or $\infty$; we obtain the same results using slightly different computation. Using $\xi_{i}$ rather than
$\zeta_{i}$, we clear denominators to obtain $$f_{m-1}(\xi_{i})g_{n-1}(\xi_{i}) = f_{n-1}(\xi_{i})g_{m-1}(\xi_{i})$$

The left-hand side has degree $e(d^{m-1} + d^{n-1} - e^{n-1})$ and the right-hand side has degree $e(d^{m-1} + d^{n-1} - e^{m-1})$. Since $n > m$, the top-degree
term comes entirely from the right-hand side, and is equal to $\xi_{i}^{d^{n-1} + d^{m-1} - e^{m-1}}$ since $\xi_{i}$ is already homogeneous. We write $D = d^{n-1}
+ d^{m-1} - e^{m-1}$.

In other words, the generators of $I'_{m, n}$ are $\xi_{i}^{D}$ over all $i$. Now $$M' = S/(S\cap I_{Y}, \xi_{1}^{D}, \ldots, \xi_{2d-e-1}^{D})$$ and since
$$S = k[\xi_{1}, \ldots, \xi_{2d-e-1}]$$ this is manifestly a finite $k$-module.\end{proof}

\section{Integrality and Critical Integrality}\label{integrality}

The methods used to prove Theorem~\ref{isovalue} can also be used to prove an integrality result:

\begin{cor}\label{integral1}In Section~\ref{finiteness}, we can work over $\mathbb{Z}$ and obtain a finite $\mathbb{Z}$-module in the statement of
Theorem~\ref{isovalue}. In other words, if we redefine $R$ as $$\mathbb{Z}[\zeta_{1}, \ldots, \zeta_{2d-1-e}, \xi_{1}, \ldots, \xi_{2d-1-e}, a_{d-1}, \ldots, a_{1},
b_{d-e-1}, \ldots, b_{0}]$$ and $S$ as $$\mathbb{Z}[\xi_{1}, \ldots, \xi_{2d-1-e}]$$ without changing the grading or the defining generators for $I_{Y}$ and $I_{m,
n}$, then the module $M'$ is a finite $\mathbb{Z}$-module.\end{cor}

\begin{proof}Nowhere in the proof of Theorem~\ref{isovalue} do we invert elements of $k$. Therefore, we get $$M' = \mathbb{Z}[\xi_{1}, \ldots,
\xi_{2d-e-1}]/(\xi_{1}^{D}, \ldots, \xi_{2d-e-1}^{D})$$ which is finite as a $\mathbb{Z}$-module.\end{proof}

On less than careful reading, it may appear as if this implies that all PCF maps with a superattracting cycle are $p$-integral, i.e. have good reduction mod $p$,
wherever their reduction (good or bad!) is $p$-tamely ramified. However, in personal communication Rivera-Letelier gave,

\begin{ex}\label{jrl}The map $$\varphi(z) = -45\frac{3z + 5}{z^{2}(z-9)}$$ is PCF but has two $5$-adically repelling fixed points, of $5$-adic absolute values $1$
and $1/5$. Since this map's degree is only $3$, this cannot possibly come from wildly ramified reduction.\end{ex}

In fact, we can instead conclude a weaker result, regarding the integrality of the critical values. But first,

\begin{defn}\label{atr}Let $\varphi$ be defined over a local field $K$. We say that $\varphi$ is \textbf{absolutely tamely ramified} if, for each integral model of
$\varphi$ over $\mathcal{O}_{K}$, the reduction mod the maximal ideal is tamely ramified after clearing common factors of the reductions of $f$ and $g$. If
$\varphi$ is defined over a global field, we say it is absolutely tamely ramified at a place if it is absolutely tamely ramified when we regard it as a map over the
completion.\end{defn}

\begin{rem}An equivalent definition for absolute tame ramification is that the local degree at each $\mathbb{P}^{1}$-disk is not divisible by the residue
characteristic of $K$. In particular, since composing maps multiplies local degrees, the composition of two absolutely tamely ramified maps, in particular the
iterate of any absolutely tamely ramified map, is again absolutely tamely ramified.\end{rem}

When $\varphi$ is absolutely tamely ramified, we can talk about configurations of points more freely, since under any $\PGL(2,K)$-conjugate the map still has tamely
ramified reduction. Of course this reduction may be as bad as a degenerate constant map, but it will still have $2d-2$ critical points and values that we can map
the critical points and values of $\varphi$ to. Using Definition~\ref{atr}, we can obtain a more precise formulation of Corollary~\ref{integralmain}:

\begin{cor}\label{integral2}Let $\varphi(z)$ be PCF, defined over a local field, absolutely tamely ramified, and with a superattracting cycle. Then the reduction of
$\varphi$ modulo the maximal ideal has an integral configuration of critical values.\end{cor}

\begin{proof}Let $K$ be the local field of definition and $\mathcal{O}_{K}$ its ring of integers. Then $\varphi \in M' \otimes \mathcal{O}_{K}$, a finite
$\mathcal{O}_{K}$-module. Thus the critical values of $\varphi$, i.e. $\xi_{1}, \ldots, \xi_{2d-e-1}$, define a finite $\mathcal{O}_{K}$-module, and in particular
are $\mathcal{O}_{K}$-integral.\end{proof}

\begin{rem}As with Definition~\ref{atr}, Corollary~\ref{integral2} also applies to maps defined over global fields.\end{rem}

In Lemma~\ref{crit}, a configuration of critical values in characteristic $p$ may not correspond to a cover of $\mathbb{P}^{1}$ because the lift of a tamely
ramified cover of $\mathbb{P}^{1}$ to characteristic $0$ may have bad reduction; on the level of groups, the tame fundamental group is the prime-to-$p$ profinite
completion of the free group on $\#\{\xi_{1}, \ldots, \xi_{2d-e-1}\}$ generators, whereas the characteristic-$0$ fundamental group is the full profinite completion,
which has more quotients and thus more actions. The interpretation of the interplay between Corollary~\ref{integral2} and the existence of non-integral PCF maps as
in Example~\ref{jrl} is that those maps are precisely the bad reduction lifts that prevent Riemann existence from holding verbatim in characteristic $p$.

For the same reason, we cannot count PCF maps in characteristic $p$ in the same way as maps in characteristic $0$. In our case of interest---that is, tamely
ramified maps with a superattracting cycle---we can lift each characteristic $p$ PCF configuration of critical values to characteristic $0$ and then lift all
characteristic $p$ PCF maps, but in the opposite direction we may encounter bad reduction. In particular, there are fewer characteristic $p$ PCF maps than
characteristic $0$ ones.

The situation for polynomials is much simpler. Not only do we not need $S$ in the proof of Theorem~\ref{isovalue}, but also we obtain:

\begin{thm}Let $f(z)$ be an absolutely tamely ramified PCF polynomial. Then $f$ has good reduction.\end{thm}

\begin{proof}See Theorem 7.1 in~\cite{Lev3}. The theorem's statement is weaker, assuming that $f$ is of degree $d < p$ or is the composition of maps of degree $d <
p$, but this assumption is only used to establish that the local degrees on certain disks are never divisible by $p$, and for that it suffices to assume absolute
tame ramification.\end{proof}

\bibliographystyle{amsplain}
\bibliography{rigidity}

\providecommand{\bysame}{\leavevmode\hbox to3em{\hrulefill}\thinspace}
\providecommand{\MR}{\relax\ifhmode\unskip\space\fi MR }
\providecommand{\MRhref}[2]{%
  \href{http://www.ams.org/mathscinet-getitem?mr=#1}{#2}
}
\providecommand{\href}[2]{#2}
\begin{thebibliography}{10}

\bibitem{AHM}
Wayne Aitken, Farshid Hajir, and Christian Maire, \emph{Finitely ramified
  iterated extensions}, Int. Math. Res. Not. \textbf{14} (2005), 855--880.

\bibitem{Lev3}
Robert Benedetto, Patrick Ingram, Rafe Jones, and Alon Levy, \emph{Critical
  orbits and attracting cycles in $p$-adic dynamics}, arXiv:1201.1605, Sep
  2012.

\bibitem{BJ1}
Nigel Boston and Rafe Jones, \emph{Arboreal {G}alois representations}, Geom.
  Dedicata \textbf{124} (2007), no.~1.

\bibitem{BJ2}
\bysame, \emph{The image of an arboreal {G}alois representation}, Pure and
  Applied Mathematics Quarterly \textbf{5} (2009), no.~1, 213--225.

\bibitem{BBLPP}
Eva Brezin, Rosemary Byrne, Joshua Levy, Kevin Pilgrim, and Kelly Plummer,
  \emph{A census of rational maps}, Conformal Geometry and Dynamics \textbf{4}
  (2000), 35--74.

\bibitem{BEKP}
Xavier Buff, Adam Epstein, Sarah Koch, and Kevin Pilgrim, \emph{On {T}hurston's
  pullback map}, Complex Dynamics: Families and Friends (Dierk Schleicher,
  ed.), A K Peters, Wellesley, MA, 2009, pp.~561--583.

\bibitem{DH}
Adrien Douady and John~H. Hubbard, \emph{A proof of {T}hurston's topological
  characterization of rational functions}, Acta Math. \textbf{171} (1993),
  263--297.

\bibitem{Eps2}
Adam Epstein, \emph{Infinitesimal {T}hurston rigidity and the
  {F}atou-{S}hishikura inequality}, arXiv:math/9902158v1, 1999.

\bibitem{Eps}
\bysame, \emph{Integrality and rigidity for postcritically finite polynomials},
  arXiv:1010.2780, 2010.

\bibitem{HT}
Benjamin Hutz and Michael Tepper, \emph{Multiplier spectra and the moduli space
  of degree 3 morphisms on $\mathbb{P}^{1}$}, arXiv:1110.5082, Oct 2011.

\bibitem{Ing}
Patrick Ingram, \emph{A finiteness result for post-critically finite
  polynomials}, arXiv:1010.3393v2; to appear in Int. Math. Res. Not., Feb 2011.

\bibitem{Koc}
Sarah Koch, \emph{Teichm\"{u}ller theory and critically finite endomorphisms},
  http://www.math.harvard.edu/~kochs/endo.pdf, 2011.

\bibitem{Lev1}
Alon Levy, \emph{The space of morphisms on projective space}, Acta Arith.
  \textbf{146} (2011), 13--31.

\bibitem{McM}
Curtis~T. McMullen, \emph{Families of rational maps and iterative root-finding
  algorithms}, Ann. of Math. (2) \textbf{125} (1987), no.~3, 467--493.
  \MR{MR890160 (88i:58082)}

\bibitem{PST}
Clayton Petsche, Lucien Szpiro, and Michael Tepper, \emph{Isotriviality is
  equivalent to potential good reduction for endomorphisms of $\mathbb{P}^{n}$
  over function fields}, Journal of Algebra \textbf{322} (2009), 3345--3365.

\bibitem{Sil96}
Joseph~H. Silverman, \emph{The space of rational maps on {$\bold P\sp 1$}},
  Duke Math. J. \textbf{94} (1998), no.~1, 41--77. \MR{MR1635900 (2000m:14010)}

\bibitem{ADS}
\bysame, \emph{The arithmetic of dynamical systems}, Graduate Texts in
  Mathematics, no. 241, Springer-Verlag, New York, 2007.

\bibitem{Sza}
Tam\'{a}s Szamuely, \emph{Galois groups and fundamental groups}, first ed.,
  Cambridge Studies in Advanced Mathematics, vol. 117, Cambridge University
  Press, Cambridge, 2009.

\end{thebibliography}

\bigskip\noindent \sc{Alon Levy---Department of Mathematics, UBC, Vancouver, BC, Canada}

\noindent \tt{email: levy@math.ubc.ca}

\end{document}